\documentclass[11pt, a4paper]{amsart}

\usepackage[colorlinks=true, linkcolor=blue,anchorcolor=blue, citecolor=green]{hyperref}

\usepackage{graphics,epic}
\usepackage{amsmath,amssymb, amsthm,mathrsfs}
\usepackage[all,2cell]{xy}
\usepackage{shorttoc}

\newtheorem{theorem}{Theorem}[section]
\newtheorem*{theorem*}{Theorem}

\newtheorem{lemma}[theorem]{Lemma}

\newtheorem{corollary}[theorem]{Corollary}

\newtheorem*{conjecture*}{Conjecture}

\newtheorem{remark}[theorem]{Remark}

\newtheorem{thm}[theorem]{Theorem}
\newtheorem{lem}[theorem]{Lemma}

\newtheorem{cor}[theorem]{Corollary}

\newcommand{\ie}{{\em i.e.}\ }

\newcommand{\opname}[1]{\operatorname{\mathsf{#1}}}

\renewcommand{\mod}{\opname{mod}\nolimits}

\newcommand{\per}{\opname{per}\nolimits}
\newcommand{\add}{\opname{add}\nolimits}

\newcommand{\der}{\md}

\newcommand{\Sub}{\opname{Sub}}

\newcommand{\X}{\mathbb{X}}

\newcommand{\D}{\mathbb{D}}

\renewcommand{\P}{\mathbb{P}}

\newcommand{\ra}{\rightarrow}

\newcommand{\Hom}{\opname{Hom}}

\newcommand{\coh}{\opname{coh}}
\newcommand{\End}{\opname{End}}

\newcommand{\Tor}{\opname{Tor}}

\newcommand{\cc}{{\mathcal C}}

\newcommand{\co}{{\mathcal O}}

\newcommand{\mc}{\mathcal{C}}
\newcommand{\md}{\mathcal{D}}

\newcommand{\tauni}{\tau^{-1}}

\renewcommand{\widetilde}[1]{\widehat{#1}}

\numberwithin{figure}{section}

\begin{document}

\title{On cluster categories of weighted projective lines with at most three weights}\thanks{Partially supported by the National Natural Science Foundation of China (Grant No. 11971326) }

\author{Changjian Fu}
\address{Changjian Fu\\Department of Mathematics\\SiChuan University\\610064 Chengdu\\P.R.China}
\email{changjianfu@scu.edu.cn}
\author{Shengfei Geng}
\address{Shengfei Geng\\Department of Mathematics\\SiChuan University\\610064 Chengdu\\P.R.China}
\email{genshengfei@scu.edu.cn}
\maketitle

\begin{abstract}
Let $\X$ be a weighted projective line and $\mc_\X$ the associated cluster category. It is known that $\mc_\X$ can be realized as a generalized cluster category  of quiver with potential. In this note, under the  assumption that $\X$ has at most three weights or is of tubular type, we prove that if the generalized cluster category $\mc_{(Q,W)}$ of a Jacobi-finite non-degenerate  quiver with potential $(Q,W)$ shares a $2$-CY tilted algebra with $\mc_\X$, then $\mc_{(Q,W)}$ is triangle equivalent to $\mc_\X$. As a byproduct, a $2$-CY tilted algebra of $\mc_\X$ is determined by its quiver provided that $\X$ has at most three weights.
To this end, for any weighted projective line $\X$ with at most three weights, we also obtain a realization of $\mc_\X$ via Buan-Iyama-Reiten-Scott's construction of $2$-CY categories arising from preprojective algebras .
\end{abstract}

\tableofcontents
\section{Introduction}
 Motivated by the theory of cluster algebras \cite{FZ1}, there has recently been much interest centered around $\Hom$-finite $2$-Calabi-Yau (2-CY for short) triangulated  categories. Many  kinds of $\Hom$-finite $2$-CY  triangulated categories with cluster-tilting objects have been investigated, for example 
\begin{enumerate}
\item[$\bullet$] the acyclic cluster category $\mc_Q$ for a finite acyclic quiver $Q$ introduced by Buan-Marsh-Reineke-Reiten-Todorov \cite{BMRRT};
\item[$\bullet$] $2$-CY triangulated categories arising from preprojective algebras introduced by  Gei\ss-Leclerc-Schr\"{o}er  \cite{GLS06, GLS10} and Buan-Iyama-Reiten-Scott \cite{BIRSc};
\item[$\bullet$] the generalized cluster category $\mc_A$ for a finite dimensional algebra $A$ of global dimension $\leq 2$  and the generalized cluster category $\mc_{(Q,W)}$ for a quiver with potential $(Q,W)$ constructed by Amiot \cite{A09};
\item[$\bullet$] the cluster category $\mc_\X$ for a weighted projective line $\X$ studied by  Barot-Kussin-Lenzing  \cite{BKL}.
\end{enumerate}
It  is known that the acyclic cluster category $\mc_Q$ can be realized as a $2$-CY triangulated category arising from preprojective algebras \cite{GLS10}, while $2$-CY triangulated categories arising from preprojective algebras  and cluster categories of  weighted projective lines  can be realized  as  generalized cluster categories  of  finite dimensional algebras of global dimension $\leq 2$ \cite{A09,ART}. Furthermore,  the generalized cluster category $\mathcal{C}_A$ can be realized as the generalized cluster category $\mathcal{C}_{(Q,W)}$ for a quiver with potential $(Q,W)$ \cite{K11}. However, it is not clear that whether every cluster category of weighted projective line can be realized as a $2$-CY triangulated category arising from preprojective algebras. 
It is conjectured that each algebraic $\Hom$-finite  $2$-CY category over an algebraically closed field of characteristic zero can be realized as a generalized cluster category for quiver with potential \cite{A08, Y}.  

Let $K$ be an algebraically closed field and $\mc$ a  $\Hom$-finite algebraic $2$-CY triangulated category over $K$. The endomorphism algebra $\End_\mc(T)$ of a cluster-tilting object $T$ of $\mc$ is called a {\it $2$-CY tilted algebra}. An open question in cluster-tilting theory is that whether the $2$-CY tilted algebra $\End_\mc(T)$ determines the $2$-CY triangulated category $\mc$ up to triangle equivalence?
A remarkable progress has been made by Keller-Reiten \cite{KR}. Namely, if $\End_\mc(T)$ is isomorphic to the path algebra of a finite acyclic quiver $Q$, then $\mc$ is triangle equivalent to the acyclic cluster category $\mc_Q$. We also refer to \cite{A09, ART, AIR15} for some other progress on this question.
Inspired by Keller-Reiten's result \cite{KR} and Happel's classification theorem \cite{H01}, it is reasonable to conjecture that if $\End_\mc(T)$ is isomorphic to a $2$-CY tilted algebra of $\mc_\X$ for a weighted projective line $\X$, then $\mc\cong \mc_\X$.
The main result of this note provides a partial evidence of a  positive answer for the aforementioned question.
  \begin{theorem}[Theorem \ref{t:main results 2}+Theorem \ref{t:main results 3}]~\label{t:main-1}
  Let $\X$ be a weighted projective line with at most three weights or of tubular type and $\mc_\X$ the cluster category of $\X$. Let $(Q,W)$ be a Jacobi-finite quiver with non-degenerate potential and $\mc_{(Q,W)}$ the associated generalized cluster category. If there is a cluster-tilting object $T$ of $\mc_{(Q,W)}$ such that $\End_{\mc_{(Q,W)}}(T)$ is isomorphic to a $2$-CY tilted algebra of the cluster category $\mc_\X$, then $\mc_{(Q,W)}$ is triangle equivalent to $\mc_\X$.
  \end{theorem}
  
  We also consider the relation between cluster categories of weighted projective lines  and $2$-CY triangulated categories arising from preprojective  algebras.  
  \begin{theorem}[Theorem~\ref{t:weighted via preprojective}]~\label{t:main-2}
  Let $\X$ be a weighted projective line with at most three weights. The cluster category $\mc_\X$ is triangle equivalent to a $2$-CY triangulated category arising from preprojective algebra.
  \end{theorem}

The paper is structured as follows.  Section~\ref{s:preliminaries} provides the required background from cluster-tilting theory, quivers with potentials and generalized cluster categories. In Section \ref{s:weighted-projective-line},  we recollect basic properties for $\mc_\X$. For a given basic cluster-tilting object $T\in \mc_\X$, there is a canonical potential $W_T$ on the Gabriel quiver $Q_T$ of $\End_{\mc_\X}(T)$ via Keller's construction. We prove that the potential $W_T$ is non-degenerate.  Assume that $\X$ has at most three weights, we further prove that $W_T$ is the unique non-degenerate potential on $Q_T$ up to right equivalence in Section~\ref{ss:uniqueness}. The uniqueness plays a key role in the proof of Theorem \ref{t:main-1} in Section~\ref{s:main-result} and Theorem \ref{t:main-2} in Section \ref{s:main-result-2}.

Throughout this paper, let $K$ be an algebraically closed field. For a differential graded(=dg) $K$-algebra $B$, denote by $\der(B)$ the derived category of dg $B$-modules. Let  $\der^b(B)$ be the full subcategory of  $\der(B)$  formed by the dg $B$-modules whose homology is of finite total dimension over $K$. Denote by $\per B$ the perfect derived category of $B$, \ie the thick subcategory of $\der(B)$ generated by $B$. Let $\mc$ be a category and $M\in \mc$, denote by $\add M$ the subcategory of $\mc$ consisting of objects which are finite direct sum of direct summands of $M$.

\section{Preliminaries}\label{s:preliminaries}
\subsection{2-CY triangulated category and cluster-tilting objects}
Let $\mc$ be a $\Hom$-finite $K$-linear triangulated category with suspension functor $[1]$.  The category $\mc$ is  {\it $2$-Calabi-Yau} ( $2$-CY for short)  if  there exist bifunctorial isomorphisms 
\[\Hom_{\mc}(X, Y[1])\cong\D\Hom_{\mc}(Y,  X[1])\] 
for arbitrary $X,Y\in\mc$,
where $\D=\Hom_K(-, K)$ is  the duality over $K$.

Let  $\mc$ be a $2$-CY triangulated category.
An  object $T\in\mc$ is a {\it cluster-tilting object} of $\mc$ if for any object $X\in\mc$ with $\Hom_{\mc}(T, X[1])=0$, we have $X\in\add T$.
Let $T=\bigoplus_{i=1}^nT_i$ be a basic cluster-tilting object in $\mc$ and $\End_{\mc}(T)$ the endomorphism algebra of $T$.
Let $Q_T$ be the quiver of $T$, whose vertices correspond to the indecomposable direct summands of $T$ and the number of arrows from $T_i$ to $T_j$ is given by the dimension of the space of irreducible maps from $T_i$ to $T_j$. 
 It is known that the quiver $Q_T$ coincides with the Gabriel quiver of $\End_\mc(T)$.

For a given basic cluster-tilting object $T=\overline{T}\oplus T_k$ in $\cc$ with indecomposable direct summand $T_k$,  there exists a unique indecomposable object $T_k^*\in\mc$ such that $T_k^*\not\cong T_k$ and $\mu_{T_k}(T):=\overline{T}\oplus T_k^*$ is also a cluster-tilting object in $\mc$~ \cite{IY}.  We call $\mu_{T_k}(T)$  the {\it mutation  of $T$ at $T_k$}.  For a vertex $k\in Q_T$, we also denote $\mu_k(T)$  the mutation of $T$ at $T_k$.  A cluster-tilting object $T'$ is {\it reachable} from $T$, if we can obtain $T'$ from $T$ by a finite sequence of mutations. The {\it cluster-tilting graph} $\mathcal{G}_{ct}(\mc)$ of $\mc$ has as vertices the isomorphism classes of basic tilting objects of $\mc$, while two vertices $T$ and $T'$ are connected by an edge if and only if $T'$ is a mutation of $T$. For any basic cluster-tilting objects $T$ and $T'$, it is clear that $T'$ is reachable from $T$ if and only if they belong to the same connected component of $\mathcal{G}_{ct}(\mc)$.

\subsection{Quivers with potentials}
We follow \cite{DWZ,KY}.
Let $Q=(Q_0,Q_1)$ be a quiver, where $Q_0$ is the set of vertices and $Q_1$ is the set of arrows. Let $KQ$ be the path algebra of $Q$ over $K$.  The {\it complete path algebra} $K\langle\langle Q\rangle\rangle$ of $Q$ is the completion of $KQ$ with respect to the ideal  generated by the arrows of $Q$.  Let $\mathfrak{m}$ be the ideal of $K\langle\langle Q\rangle\rangle$ generated by arrows of $Q$. 
In particular, $K\langle\langle Q\rangle\rangle$ is a topological algebra via $\mathfrak{m}$-adic topology.
An element $W\in K\langle\langle Q\rangle\rangle$ is a {\it potential} on $Q$ if $W$ is a (possibly infinite) linear combination of cycles in $Q$.
Two potentials $W$ and $W'$ are {\it cyclically equivalent} if $W-W'$ lies in the closure  of the span of all elements of the form $\alpha_s\cdots\alpha_2\alpha_1-\alpha_1\alpha_s\cdots\alpha_2$, where $\alpha_s\cdots\alpha_2\alpha_1$ is a cyclic path.

Let $W$ be a potential on $Q$ such that $W\in \mathfrak{m}^2$  and no two cyclically equivalent
cyclic paths appear in the decomposition of $W$. Then the pair $(Q, W)$ is called a {\it quiver with potential} (QP for short). 
Two QPs $(Q,W)$ and $(Q',W')$ are {\it right equivalent} if $Q$ and $Q'$ have the same set of vertices and there exists an algebra isomorphism $\varphi: K\langle\langle Q\rangle\rangle\ra K\langle\langle Q'\rangle\rangle$ whose restriction on vertices is the identity map and $\varphi(W)$ and $W'$ are cyclically equivalent.

For every arrow $\alpha\in Q_1$, the {\it cyclic derivative} $\partial_\alpha$ is the continuous $K$-linear map from the closed subspace of $K\langle\langle Q\rangle\rangle$ generated by all cyclic paths of $Q$ to $K\langle\langle Q\rangle\rangle$ acting on cyclic paths by
 \[\partial_\alpha(\alpha_1\cdots\alpha_m)=\sum_{p:\alpha_p=\alpha}\alpha_{p+1}\cdots\alpha_m\alpha_1\cdots\alpha_{p-1}.\]
 Let $(Q,W)$ be a quiver with potential.
Denote by $J(W)$ the closure of the two-sided ideal of $K\langle\langle Q\rangle\rangle$ generated by  the elements $\partial_\alpha W$ for all $\alpha\in Q_1$.
The {\it Jacobian algebra} of $(Q,W)$ is 
\[J(Q,W):=K\langle\langle Q\rangle\rangle/J(W).\]
A QP $(Q,W)$ is {\it Jacobi-finite} if the Jacobian algebra $J(Q,W)$ is finite-dimensional. 
A QP $(Q,W)$ is {\it reduced} if $\partial_aW\in \mathfrak{m}^2$ for all arrows $\alpha\in Q_1$. For a reduced Jacobi-finite quiver with potential  $(Q,W)$, it is known that the Gabriel quiver of $J(Q,W)$ is $Q$.

\subsection{QP-mutation and non-degenerate potentials}
Let $(Q,W)$ be a quiver with potential and $k$ a vertex of $Q$.  Assume that $(Q,W)$ satisfies the following conditions:
\begin{itemize}
\item[(c1)] $Q$ has no loops;
\item[(c2)] $Q$ has no $2$-cycles at $k$;
\item[(c3)] no cyclic path occurring in the expansion of $W$ starts and ends at vertex $k$.
\end{itemize}
Note that under the condition $(c1)$, any potential is cyclically equivalent to a potential satisfying $(c3)$.
Derksen-Weyman-Zelevinsky~\cite{DWZ} defined an operation $\mu_k$ on $(Q,W)$, called the {\it $QP$-mutation at vertex $k$}. The resulting $\mu_k(Q,W)$ is a reduced quiver with potential. We refer to \cite[Section 5]{DWZ} and \cite[Section 2.4]{KY} for the definition.

\begin{lem}\cite[Thm 5.2]{DWZ}\label{right equivalence determined by mutation}
Let $(Q,P)$ be a quiver with potential and $k$ a vertex of $Q$ such that $\mu_k(Q,W)$ is well-defined. Then the right-equivalence class of $\mu_k(Q,W)$ is determined by the right-equivalence class of $(Q,W)$.
\end{lem}

 If $Q$ has no loops nor $2$-cycles, then $\mu_k(Q,W)$ is well-defined for any vertex $k$ of $Q$. In this case, the quiver with potential $(Q,W)$ is {\it $2$-acyclic}. A potential $W$ is {\it non-degenerate} on $Q$, if for any sequence $i_1,\cdots, i_t$ of vertices, all QPs $\mu_{i_1}(Q,W)$, $\mu_{i_2}\mu_{i_1}(Q,W)$,$\ldots$, $\mu_{i_t}\cdots \mu_{i_2}\mu_{i_1}(Q,W)$ are $2$-acyclic. We also call $(Q,W)$ a non-degenerate QP.
A potential $W$ on $Q$ is {\it rigid} if every cycle in $Q$ is cyclically equivalent to an element of the Jacobian  ideal $J(W)$. According to \cite[Cor 8.2]{DWZ}, every rigid potential is non-degenerate.

Let $(Q,W)$ be a quiver with potential.
Let $I$ be a subset of  $Q_0$. Let $(Q|_I, W|_I)$ be the restriction of $(Q,W)$ to $I$.
\begin{lem}\cite[Prop 8.9]{DWZ}\cite[Cor 22]{F}\label{keep non-degenerate}
 If $(Q,W)$ is non-degenerate (resp. rigid), then $(Q|_{I},W|_{I})$ is non-degenerate (resp. rigid).
\end{lem}

\subsection{QP-mutation and quiver mutation}
 Let $Q$ be a finite quiver without loops nor $2$-cycles and $k$ a vertex. The \textit{quiver mutation} $\mu_k(Q)$ of $Q$ at vertex $k$ is the  quiver obtained from $Q$ as follows:
 	\begin{itemize}
 		\item for each subquiver $i\xrightarrow{\beta}k\xrightarrow{\alpha}j$, we add a new arrow $[\alpha\beta]:i\to j$;
 		\item we reverse all arrows with source or target $k$;
 		\item we remove the arrows in a maximal set of pairwise disjoint $2$-cycles.
 		\end{itemize}
Quiver mutations have close relation with QP-mutations for non-degenerate QPs. 		
 		\begin{lem}\cite[Prop 7.1]{DWZ}\label{correspondence with mutation}
Let $(Q,W)$  be a non-degenerate quiver with potential.  Suppose that $(Q',W')=\mu_i(Q,W)$ for some vertex $i$, then $Q'=\mu_i(Q)$.
\end{lem}

 If $Q'$ can be obtained from $Q$ by a finite sequence of quiver mutations,  then  $Q'$ and $Q$ are {\it mutation equivalent}.  The following is well-known (cf. \cite[Lem 9.2]{GLaS} for instance).
 \begin{lemma}~\label{l:glas}
 Assume that $Q$ is mutation equivalent to an acyclic quiver. Then there exists a unique non-degenerate potential on $Q$ up to right equivalence.
\end{lemma}

\subsection{Generalized cluster category $\mc_{(Q,W)}$}
Let $(Q,W)$ be a  quiver with potential, 
 denote by $\Gamma(Q,W)$ the complete Ginzburg dg algebra, we refer to \cite[Section 4.2]{G} and \cite[Section 2.6]{KY} for the precisely construction. It is known that $\Gamma(Q,W)$ is a non positive dg algebra and the  Jacobian algebra $J(Q,W)$ is the zeroth homlogy of  $\Gamma(Q,W)$. 
Keller\cite{K11} proved that the perfect derived category $\per \Gamma(Q,W)$ contains $\der^b(\Gamma(Q,W))$.
The {\it generalized cluster category} $\mc_{(Q,W)}$ associated
to $(Q,W)$ is defined to be the  Verdier quotient \[\mc_{(Q,W)}:=\per\Gamma(Q,W)/\md^b(\Gamma(Q,W)).\] 
In general, $\mc_{(Q,W)}$ has infinite dimensional $\Hom$-space.

\begin{lem}\cite[Lem 2.9]{KY}\label{l:KY1}
Let $(Q,W)$ and $(Q',W')$ be two quivers with potentials.
If $(Q,W)$ and $(Q',W')$ are right equivalent, then the complete Ginzburg dg
algebras $\Gamma(Q,W)$ and $\Gamma(Q',W')$ are isomorphic to each other. In particular, $\mc_{(Q,W)}$ is triangle equivalent to $\mc_{(Q',W')}$.
\end{lem}

\begin{lem}\cite[Thm 3.5]{A09}\label{Hom-finite}
 Let $(Q, W)$ be a Jacobi-finite quiver with potential. The generalized cluster category $\mc_{(Q,W)}$ associated to $(Q, W)$ is $\Hom$-finite and 2-CY. Moreover, the image  of the free module $\Gamma(Q,W)$ in  $\mc_{(Q,W)}$ is a cluster-tilting object. Its endomorphism algebra is isomorphic to the Jacobian algebra $J(Q,W)$.
\end{lem}

Let $(Q,W)$ be a Jacobi-finite quiver with potential. By Lemma \ref{Hom-finite}, $\Gamma(Q,W)$ is a cluster-tilting object in the $2$-Calabi-Yau triangulated category $\mc_{(Q,W)}$.  Therefore, we can form the mutation of $\Gamma(Q,W)$ at any indecomposable direct summand. In particular, for each vertex $i$, we have the mutation  $\mu_i(\Gamma(Q,W))$ of $\Gamma(Q,W)$ at $e_i\Gamma(Q,W)$, where $e_i$ is the primitive idempotent associated to $i$.

The following is a direct consequence of \cite[Thm 3.2]{KY} and Lemma \ref{Hom-finite}.
\begin{lemma}\label{l:KY2}
Let $(Q,W)$ be a  Jacobi-finite quiver with potential and $i$ a vertex of $Q$. Assume that $Q$ has no loops nor $2$-cycles. 
There is a triangle equivalence from $\mc_{(Q,W)}$ to $\mc_{\mu_i(Q,W)}$ which sends $\mu_i(\Gamma(Q,W))$ to $\Gamma(\mu_i(Q,W))$. Consequently, $\End_{\mc_{(Q,W)}}(\mu_i(\Gamma(Q,W)))\cong J(\mu_i(Q,W))$.
\end{lemma}

\subsection{Generalized cluster categry $\mc_A$}\label{s:gen-cluster-category}
Let $A$ be a finite dimensional $K$-algebra with finite global dimension. Denote by $\der^b(A)$ the bounded derived category of finitely generated right $A$-modules and $\nu_A$ a Serre functor of $\der^b(A)$.
Let $B$ be the dg algebra $A\oplus \D A[-3]$ with trivial differential and  $p:B\ra A$ the canonical projection. 
The projection $p$ induces a triangle functor \[p_*:\md^b(A)\ra \md^b (B).\]
 Let $\langle A\rangle_B$ be the thick subcategory of $\md^b(B)$ generated by the image of $p_*$. It is known that $\per B\subset \langle A\rangle_B$ and we can form the Verdier quotient\[\mc_A:=\langle A\rangle_B/\per B.\]
  The category $\mc_A$ is called the {\it generalized cluster category} of $A$. By \cite[Thm 7.1]{K05}, the triangulated hull of the orbit category $\der^b(A)/\nu_A[-2]$ is  $\mc_A$. If $A=KQ$ for a finite acyclic quiver $Q$, then $\mc_{KQ}=\der^b(A)/\nu_A[-2]$ by \cite[Thm 4]{K05}. In this case, $\mc_{KQ}$ is the {\it cluster category} of $KQ$ introduced by \cite{BMRRT}. In general, $\mc_A$ has infinite dimensional $\Hom$-space.

\begin{thm}\cite[Thm 4.10]{A09}
Let $A$  be a finite-dimensional $K$-algebra of  global dimension $\leq 2$.  If the functor $\Tor^A_2 (?, \D A)$ is nilpotent,  then $\mc_A$ is a $\Hom$-finite $2$-CY triangulated category and the object $A$ is a cluster tilting object in $\mc_A$.
\end{thm}

\subsection{Quiver with potential for algebra of global dimension $\leq 2$}\label{s:keller}

Let $A=KQ_A/I$ be a finite dimensional algebra of global dimension $\leq 2$, where $I$ is  an admissible ideal of the path algebra $KQ_A$.  Keller~\cite{K11} introduced a quiver with potential $(\tilde{Q}_A,W_A)$ associated to $A$.
Fix a set of representatives of minimal relations of $I$. For each such representative $r$ which starts at vertex $i$ and ends at vertex $j$, let $\rho_r$ be a new arrow from $j$ to $i$. The quiver $\tilde{Q}_A$ is obtained from $Q$ by adding all the arrows $\rho_r$.  The potential $W_A$ on $\tilde{Q}_A$ is given by
$W_A = \sum_{r} r\rho_r$, where the sum ranges over  the set of representatives.

\begin{lemma}\cite[Thm 6.12]{K11}\label{l:K11}
The generalized cluster category $\mc_A$ is triangle equivalent to the generalized cluster category $\mc_{(\tilde{Q}_A,W_A)}$. 
Moreover, the endomorphism algebra $\End_{\mc_A} (A)$ is isomorphic to the  Jacobian algebra $J(\tilde{Q}_A,W_A)$.
\end{lemma}

\section{Cluster categories of weighted projective lines}\label{s:weighted-projective-line}
\subsection{Weighted projective lines}
Fix a positive integer $t$.
Let $\X=\X(\mathbf{ p},\boldsymbol{\lambda})$ be a weighted projective line attached to a weight sequence  $\mathbf{ p}=(p_1,\dots,p_t)$
 of integers $p_i\ge 2$ and 
a  parameter sequence $\boldsymbol{\lambda}=(\lambda_1,\dots,\lambda_t)$  of pairwise distinct elements of $\P_1(K)$. Without loss of generality, we may assume that $\lambda_1=[1:0], \lambda_2=[0:1], \lambda_3=[1:1]$.
 Denote by $\coh \X$ the category of coherent sheaves over $\X$, which is a $\Hom$-finite hereditary abelian category with tilting objects.  We refer to \cite{GL} for basic properties of weighted projective lines.
 
Let $\mathbb{L}$ be the rank one  abelian group generated by $\vec{x}_1,\ldots, \vec{x}_t$ with the relations
\[p_1\vec{x}_1=p_2\vec{x}_2=\cdots=p_t\vec{x}_t=:\vec{c},
\]
where  $\vec{c}$  is called the {\it canonical element} of $\mathbb{L}$.
For each coherent sheaf $E$ over $\X$ and $\vec{x}\in \mathbb{L}$, denote by $E(\vec{x})$ the  grading shift of $E$ with respect to $\vec{x}$. Denote by $\mathcal{O}$ the structure  sheaf of $\X$. It is known that each line bundle is given by the grading shift $\mathcal{O}(\vec{x})$ for a unique element $\vec{x}\in \mathbb{L}$. 

The category $\coh\X$ has $t$ exceptional tubes consisting of sheaves of finite length.
In the $i$-th exceptional tube of rank $p_i$,  there is a unique simple object $S_i$ such that $\Hom_{\coh\X}(\co,S_i)\neq 0$.
Note that $S_i$ is also the unique simple object in the $i$-th exceptional tube of rank $p_i$ satisfying $\Hom_{\coh\X}(\co(\vec{c}),S_i)\neq 0$.
For a positive integer $j$, denote by $S_i^{[j]}$ the unique indecomposable object lying in the same tube with $S_i$ which  has length $j$ and top $S_i$.
The object  \[T_{sq}=\co\oplus\co(\vec{c})\oplus(\bigoplus_{i=1}^t\bigoplus_{j=1}^{p_i-1} S_{i}^{[j]})\] is a basic tilting object in $\coh\X$, which is called the {\it squid tilting object} (cf. \cite[Section 8]{BKL}).

\subsection{Classifications of weighted projective lines} \label{ss:classification}

Let $p$ be the least common multiple of $p_1,\ldots, p_t$. The genus $g_{\X}$ of a weighted projective line $\X$ is defined by \[g_{\X}=1+\frac{1}{2}((t-2)p-\sum_{i=1}^tp/p_i).\]
A weighted projective line of genus $g_\X<1$($g_{\X}=1$, resp. $g_{\X}>1$) is called of {\it domestic} ({\it tubular}, resp. {\it wild}) type. It is known that a weighted projective line $\X$ is derived equivalent to a finite dimensional hereditary $K$-algebra if and only if $\X$ is of domestic type.

The domestic weight types are, up to permutation, $(q)$ with $q\ge 1$, $(q_1,q_2)$ with $q_1,q_2\ge 2$, $(2,2,n)$ with $n\ge 2$, $(2,3,3),(2,3,4),(2,3,5)$,  whereas the tubular weight types are, up to permutation, $(2,2,2,2),(3,3,3),(2,4,4)$ and $(2,3,6)$.

\subsection{Cluster category $\mc_\X$}
Let $\der^b(\coh\X)$ be the bounded derived category of $\coh\X$ with suspension functor $[1]$. Let $\tau:\der^b(\coh\X)\to \der^b(\coh\X)$ be the Auslander-Reiten translation functor.
The {\it cluster category of $\X$} is defined as the orbit category \[\mc_\X:=\md^b(\coh\X)/\tauni[1].\]  It has been proved by Keller \cite{K05} that $\mc_\X$ admits a canonical triangle structure such that the projection $\pi: \der^b(\coh\X)\to \mc_{\X}$ is a triangle functor. Moreover, $\mc_\X$ is a $\Hom$-finite $2$-CY triangulated category with cluster-tilting objects.
By \cite{BKL}, the cluster-tilting objects in $\mc_\X$ are precisely the tilting objects in $\coh\X$ and $\mc_\X$ has a cluster structure in the sense of \cite{BIRSc}. In particular, for any basic cluster-tilting object $T\in\mc_\X$, the quiver $Q_T$ has no loops nor $2$-cycles and the mutation of cluster-tilting objects is compatible with the quiver mutation, \ie $\mu_i(Q_T)=Q_{\mu_i(T)}$ for each vertex $i$ of $Q_T$ (cf. \cite[Thm 3.1]{BKL}). 

The connectedness of cluster-tilting graph $\mathcal{G}_{ct}(\mc_\X)$ has been established by \cite{BMRRT} for domestic type, by \cite[Thm 8.8]{BKL} for tubular type and by \cite[Thm 1.2]{FG} in full of generality. For a $2$-CY triangulated category which shares a $2$-CY tilted algebra with $\mc_\X$, we have

\begin{lem}\label{connectedness of mc}
Let $\mc$ be a $\Hom$-finite 2-CY triangulated category over $K$. If there is a cluster-tilting object $T$ in $\mc$ such that $\End_{\mc}(T)$ is isomorphic to a $2$-CY tilted algebra for some cluster-tilting object of $\mc_\X$,  then the cluster-tilting graph $\mathcal{G}_{ct}(\mc)$ is connected.
\end{lem}
\begin{proof}
Let $M$ be a cluster-tilting object of $\mc_\X$ such that $\End_{\mc_\X}(M)\cong \End_\mc(T)=:\Lambda$. According to \cite[Thm 4.7]{AIR}, $\Hom_{\mc}(T,-)$ induces  a bijection  between the set of basic cluster-tilting objects in $\mc$ and the set of basic support $\tau$-tilting $\Lambda$-modules, while
 $\Hom_{\mc_\X}(M,-)$ induces  a bijection  between the set of basic cluster-tilting objects in $\mc_\X$ and the set of basic support $\tau$-tilting $\Lambda$-modules.  Moreover, the bijections are compatible with mutations. Consequently, $\mathcal{G}_{ct}(\mc)\cong \mathcal{G}_{ct}(\mc_\X)$. We conclude that $\mathcal{G}_{ct}(\mc)$ is connected by \cite[Thm 1.2]{FG}.
\end{proof}

\subsection{Cluster category $\mc_\X$ as generalized cluster category}\label{s:gen-cluster-cat-tilting-quiver}

Let $T$ be a basic tilting object of $\coh\X$ and $A=\End_{\coh\X}(T)$ the endomorphism algebra. It is known that $A$ is a finite dimensional algebra of global dimension $\leq 2$ and $\der^b(\mod A)\cong \der^b(\coh\X)$. According to \cite[Thm 7.1]{K05}, we know that $\mc_A\cong\mc_\X$ and $\End_{\mc_\X}(T)\cong \End_{\mc_{A}} (A)$.  Let $(\tilde{Q}_A,W_A)$ be the quiver with potential associated to $A$ via Keller's construction (cf. Section~\ref{s:keller}). By Lemma~\ref{l:K11}, we obtain
\[\mc_{\X}\cong \mc_A\cong \mc_{(\tilde{Q}_A,W_A)}~\text{and}~\End_{\mc_\X}(T)\cong \End_{\mc_{A}} (A)\cong J(\tilde{Q}_A,W_A).
\]
We remark that the equivalence from $\mc_{\X}\xrightarrow{\sim} \mc_{(\tilde{Q}_A,W_A)}$ sends $T$ to $\Gamma(\tilde{Q}_A, W_A)$.
\begin{lem}\label{l:non-deg}
The potential $W_A$ is non-degenerate.
\end{lem}
\begin{proof}
Recall that $Q_T$ is the Gabriel quiver of $\End_{\mc_\X}(T)$.
Since $\End_{\mc_\X}(T)$ is finite dimensional, it follows that $(\tilde{Q}_A, W_A)$ is Jacobi-finite. Hence we may identify $\tilde{Q}_A$ with $Q_T$. By \cite[Thm 3.1]{BKL}, $Q_T$ has no loops nor $2$-cycles.  For each vertex $i$ of $Q_T$, by Lemma~\ref{l:KY2}, we have 
\[J(\mu_i(Q_T,W_A))\cong\End_{\mc_{(\tilde{Q}, W_A)}}(\mu_i(\Gamma(\tilde{Q}_A, W_A)))\cong \End_{\mc_\X}(\mu_i(T)).\]
In particular , $\mu_i(Q_T,W_A)$ is Jacobi-finite and hence the quiver of $\mu_i(Q_T,W_A)$ is isomorphic to $Q_{\mu_i(T)}$. Therefore $\mu_i(Q_T,W_A)$ is $2$-acyclic. Continuing this process, we conclude that $(Q_T,W_A)$ is non-degenerate.

\end{proof}
The following is a direct consequence of Lemma~\ref{l:non-deg}.
\begin{corollary}
Every $2$-CY tilted algebra of $\mc_\X$ is a Jacobian algebra for a non-degenerate quiver with potential.
\end{corollary}

\section{Cluster category $\mc_\X$ vs generalized cluster category $\mc_{(Q,W)}$}\label{s:main-result}
\subsection{Cluster category $\mc_\X$ with at most three weights}\label{ss:uniqueness}

Let $\X$ be a weighted projective line with weight sequence $(p_1,p_2, p_3)$. Recall that we have a basic tilting object $T_{sq}=\co\oplus\co(\vec{c})\oplus(\bigoplus_{i=1}^3\bigoplus_{j=1}^{p_i-1} S_{i}^{[j]})$, which induces a basic cluster-tilting object in $\mc_\X$. The Gabriel quiver $Q_{T_{sq}}$ of $\End_{\mc_\X}(T_{sq})$ is described as follows
\[
\xymatrix@R=0.3cm@C=0.4cm{
&&S_1^{[p_1-1]}\ar[r]\ar[ddl]&\cdots\ar[r]&S_1^{[1]}=S_1\\
&&S_2^{[p_2-1]}\ar[r]\ar[dl]&\cdots\ar[r]&S_2^{[1]}=S_2\\
&\co\ar@{=>}[rr]&&\co(\vec{c})\ar[dl]\ar[ul]\ar[uul]&\\
&&S_3^{[p_3-1]}\ar[r]\ar[ul]&\cdots\ar[r]&S_3^{[1]}=S_3.
}
\]

\begin{lem}\label{unique potential}
Let $\X$ be a weighted projective line with at most three weights. Let $T$ be a basic cluster-tilting object in $\mc_{\X}$ and  $Q_T$ the quiver of $\End_{\mc}T$, then there  is a unique non-degenerate potential on $Q_T$ up to right equivalence.
\end{lem}
\begin{proof}
According to Lemma~\ref{l:non-deg}, it remains to prove the uniqueness.

If $\X$ has at most two weights, then $\X$ is of  domestic type. In this case, $\coh\X$ is derived equivalent to the path algebra $KQ$ of an acyclic quiver $Q$ of affine type.   Consequently, $Q_T$ is mutation equivalent to $Q$. By Lemma~\ref{l:glas}, we conclude that $Q_T$  has a unique non-degenerate potential up to right equivalence.

Suppose that $\X$ has weight sequence $(p_1,p_2,p_3)$.  
Let $Q$ be the quiver as follows
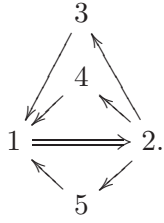
\[
\xymatrix@R=0.3cm@C=0.4cm{
&3\ar[ddl]&\\
&4\ar[dl]&\\
1\ar@{=>}[rr]&&2.\ar[ul]\ar[uul]\ar[dl]\\
&5\ar[ul]&
}
\]
It is straightforward to check that $\mu_2\mu_3\mu_4\mu_5(Q)$ is an acyclic quiver. Therefore $Q$ has a unique non-degenerate potential $S$ up to right equivalence.  Note that $Q$ is a full subquiver of ${Q}_{T_{sq}}$ and every cycle of $Q_{T_{sq}}$ lies in $Q$. If $W$ is a non-degenerate potential on ${Q}_{T_{sq}}$, then $W$ is a non-degenerate potential on $Q$ by Lemma~\ref{keep non-degenerate}. Hence, $(Q,W)$ is right equivalent to $(Q,S)$. Clearly, the right equivalence between $(Q,W)$ and $(Q,S)$ lifts to a right equivalence between $(Q_{T_{sq}}, W)$ and $(Q_{T_{sq}},S)$. Therefore $S$ is the unique non-degenerate potential on ${Q}_{T_{sq}}$ up to right equivalence.

By \cite[Thm 1.2]{FG}, $T$ is reachable from $T_{sq}$. Since the mutation of cluster-tilting objects in $\mc_\X$ is compatible with quiver mutation, it follows that there is a sequence of vertices $i_1,\ldots, i_k$ of $Q_T$ such that $Q_{T_{sq}}=\mu_{i_k}\cdots\mu_{i_2}\mu_{i_1}(Q_T)$. By Lemma~\ref{right equivalence determined by mutation} and Lemma~\ref{correspondence with mutation},  the uniqueness of non-degenerate potential on $Q_T$ follows from the uniqueness of non-degenerate potential on ${Q}_{T_{sq}}$.
\end{proof}

\begin{cor}\label{c:quiver-to-equivalent}
Let $\X$ be a weighted projective line with at most three weights. Let $Q$ be the Gabriel quiver of a $2$-CY tilted algebra of $\mc_\X$.  If  $W$ is a non-degenerate potential on $Q$, then $\mc_{(Q,W)}\cong \mc_\X$.
\end{cor}
\begin{proof}
Let $T$ be a basic cluster-tilting object of $\mc_\X$ such that $Q_T=Q$. By Lemma~\ref{l:non-deg}, there is a non-degenerate potential $W_T$ on $Q_T$ such that $\mc_\X\cong \mc_{(Q_T,W_T)}$. Now the result follows from Lemma~\ref{unique potential} and Lemma~\ref{l:KY1}.
\end{proof}

Let $T$ be a basic cluster-tilting object of the cluster category $\mc_\X$ of a weighted projective line $\X$. The $2$-CY tilted algebra $\End_{\mc_\X}(T)$ is {\it determined by its quiver} if for any basic cluster-tilting object $T'$ of the cluster category $\mc_{\X'}$ of some weighted projective line $\X'$ such that $Q_T\cong Q_{T'}$, then $\End_{\mc_\X}(T)\cong \End_{\mc_{\X'}}(T')$.

\begin{cor}
Let $\X$ be a weighted projective line with at most three weights. Then
$2$-CY tilted algebras arising from $\mc_{\X}$ are uniquely determined by their quivers.
\end{cor}
\begin{proof}
Let $T$ be a basic cluster-tilting object of $\mc_\X$ and $T'$  a cluster-tilting object of $\mc_{\X'}$ for some weighted projective line $\X'$ such that $Q_T\cong Q_{T'}$.

By Lemma ~\ref{l:non-deg}, there exist non-degenerate quivers with potentials $(Q_T, W_T)$ and $(Q_{T'},W_{T'})$ such that $\mc_{\X}\cong \mc_{(Q_T,W_T)} $ and $\mc_{\X'}\cong \mc_{(Q_{T'}, W_{T'})}$. Moreover, $\End_{\mc_\X}(T)\cong J(Q_T,W_T)$ and $\End_{\mc_{{\X'}}}(T')\cong J(Q_{T'},W_{T'})$. By Lemma~\ref{unique potential}, $(Q_T, W_T)$ and $(Q_{T'}, W_{T'})$ are right equivalent. Hence,   $\mc_{(Q_T,W_T)} \cong \mc_{(Q_{T'}, W_{T'})}$ and $J(Q_T,W_T)\cong J(Q_{T'},W_{T'})$ by Lemma~\ref{l:KY1}.

\end{proof}

\begin{theorem}\label{t:main results 2}
Let  $(Q,W)$ be a Jacobi-finite quiver with non-degenerate potential.  If there is a cluster-tilting object $T$ in $\mc_{(Q,W)}$ such that $\End_{\mc_{(Q,W)}}(T)$ is isomorphic to a  $2$-CY tilted algebra arising from some weighted projective line $\X$ with at most three weights, then
$\mc_{(Q,W)}$ is triangle equivalent to $\mc_{\X}$.
\end{theorem}
\begin{proof}

By Lemma~\ref{Hom-finite}, $\Gamma(Q,W)$ is a cluster-tilting object in $\mc_{(Q,W)}$. Consequently, $T$ is reachable from $\Gamma$ by Lemma~\ref{connectedness of mc}.
In particular, there exists  a sequence of vertices $i_1,\ldots, i_l$ of $Q$ such that $T=\mu_{i_l}\cdots\mu_{i_2}\mu_{i_1}(\Gamma(Q,W))$. Denote by $(Q',W')=\mu_{i_l}\cdots\mu_{i_2}\mu_{i_1}(Q,W)$, which is a non-degenerate quiver with potential. Applying Lemma ~\ref{l:KY2},  we have $\mc_{(Q',W')}\cong \mc_{(Q,W)}$ and $\End_{\mc_{(Q,W)}}(T)\cong J(Q',W')$. Note that $(Q,W)$ is Jacobi-finite, which implies that $(Q',W')$ is Jacobi-finite. Hence $Q'=Q_T$.
It follows from  Corollary~\ref{c:quiver-to-equivalent} that   $\mc_{(Q,W)}\cong \mc_{(Q',W')}\cong \mc_\X$.

\end{proof}

\subsection{Cluster category $\mc_\X$ of tubular type}

Let $\X$ be a weighted projective line with weight sequence $(2,2,2,2)$ and parameter sequence $\boldsymbol{\lambda}=(\lambda_1,\lambda_2,\lambda_3,\lambda_4)$. Recall that $\lambda_1=[1:0], \lambda_2=[0:1],\lambda_3=[1:1]$.  The point $\lambda_4$ is uniquely determined by $\lambda\in K\backslash\{0,1\}$ by setting $\lambda_4=[\lambda:1] $.
 Hence a weight projective line with weight type $(2,2,2,2)$ can by denote by $\X(2,2,2,2; \lambda)$ for  $\lambda\neq 0,1$.

Let $\X=\X(2,2,2,2;\lambda)$ and $\X'=X(2,2,2,2;\lambda')$ be weighted projective lines.
It is known that $\coh\X$ and $\coh\X'$ are equivalent abelian categories iff
$\md^b(\coh \X)$ and $\md^b(\coh \X')$ are triangle equivalent iff $\lambda'\in O(\lambda)=\{\lambda,\lambda^{-1},1-\lambda,1-\lambda^{-1},(1-\lambda)^{-1},\frac{\lambda}{1-\lambda}\}$. Therefore, $\mc_{\X}$ and $\mc_{\X'}$ are triangle
equivalent if $\lambda'\in O(\lambda)=\{\lambda,\lambda^{-1},1-\lambda,1-\lambda^{-1},(1-\lambda)^{-1},\frac{\lambda}{1-\lambda}\}$.

Let  $Q^{(2,2,2,2)}$ be the following quiver
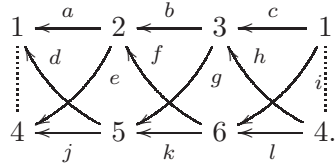
\[
\xymatrix{
1\ar@{.}[d]&2\ar[l]_a\ar@/^/[dl]^<<<{e}&3\ar[l]_b\ar@/^/[dl]^<<<g&1\ar[l]_c\ar@/^/[dl]^<<<i\ar@{.}[d]\\
4&5\ar[l]^j\ar@/^/[ul]_>>>{d}&6\ar[l]^k\ar@/^/[ul]_>>>f&4.\ar[l]^l\ar@/^/[ul]_>>>h
}
\]
For each $\lambda\in K\backslash\{0,1\}$, set
\[W_{\lambda}=\lambda abc-dgc+dki-afi+jgh-ebh+efl-jkl.\]
Let $A=kQ_A/I_A$, where $Q_A$  is  the following quiver 
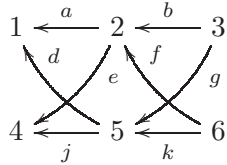
\[
\xymatrix{
1&2\ar[l]_a\ar@/^/[dl]^<<<e&3\ar[l]_b\ar@/^/[dl]^<<<g\\
4&5\ar[l]^j\ar@/^/[ul]_>>>d&6\ar[l]^k\ar@/^/[ul]_>>>f
}
\]
and $I=\langle dg-\lambda ab,dk-af,jg-eb,ef-jk\rangle,\lambda\notin\{0,1\}$.
It is straightforward to check that $A$ is a tubular algebra of type $(2,2,2,2;\lambda)$. In particular, $\der^b(\mod A)\cong \der^b(\coh\X)$ for a weighted projective line $\X$ with weight sequence $(2,2,2,2;\lambda)$. Consequently, $\mc_A\cong \mc_\X$.
Applying Keller's construction to $A$, we obtain a  quiver with potential $(\tilde{Q}_A,W_A)=(Q^{(2,2,2,2)},W_\lambda)$.
Hence,  by Lemma~\ref{l:K11}, we have $\mc_{\X}\cong\mc_A\cong\mc_{(Q^{(2,2,2,2)},W_{\lambda})}$.

\begin{lem}\cite[Prop 9.15]{GLaS}\label{potential on 2222}
\begin{itemize}
\item[(1)] Each non-degenerate potential on $Q^{(2,2,2,2)}$ is right equivalent to one of the potentials
$W_{\lambda}$ with $\lambda\notin \{0,1\}$.
\item[(2)] The Jacobian algebras $J(Q^{(2,2,2,2)},W_{\lambda})$ and $J(Q^{(2,2,2,2)},W_{\lambda'})$ are isomorphic if and only if $\lambda'\in\{\lambda,\lambda^{-1}\}$.
\end{itemize}
\end{lem}

Let $\X$ be a weighted projective line of type $(2,2,2,2;\lambda)$. 
It has been observed by ~\cite[Prop 2.5.2 (b)]{GG}  that every $2$-CY tilted algebra of $\mc_{\X}$ satisfies the vanishing condition of \cite[Thm 5.2]{BIRSm}. Consequently, we obtain
\begin{lem}\label{l:gg}
Let $\X$ be a weighted projective line of type $(2,2,2,2;\lambda)$. Let $\mc$ be a $2$-CY triangulated category over $K$ and $M\in \mc$ a basic cluster-tilting object such that $\End_\mc(M)$ is a $2$-CY tilted algebra of $\mc_\X$. If there is a potential $W$ on $Q_M$ such that $J(Q_M,W)\cong \End_\mc(M)$, then for any vertex $k$ of $Q_M$, $\End_\mc(\mu_k(M))\cong J(\mu_k(Q_M,W))$.

\end{lem}

\begin{theorem}\label{t:main results 3}
Let  $(Q,W)$ be a Jacobi-finite quiver with non-degenerate potential.  If there is a cluster-tilting object $T$ in $\mc_{(Q,W)}$ such that $\End_{\mc}(T)$ is isomorphic to a  $2$-CY tilted algebra arising from some weighted projective line $\X$  of tubular type, then $\mc$ is triangle equivalent to $\mc_{\X}$.
\end{theorem}
\begin{proof}
 According to Section~\ref{ss:classification},  $\X$ is of tubular type if and only if $\X$ has weight sequences $(3,3,3)$, $(2,4,4)$, $(2,3,6)$ or $(2,2,2,2)$.  By Theorem~\ref{t:main results 2}, it remains to consider the case $(2,2,2,2)$. 
 
 Let $\X=\X(2,2,2,2; \lambda)$ for some $\lambda\in K\backslash\{0,1\}$. Note that we have $\mc_{\X}\cong \mc_{(Q^{(2,2,2,2)}, W_\lambda)}$. In particular, there is a basic cluster-tilting object in $\mc_\X$ whose quiver is $Q^{(2,2,2,2)}$. Recall that mutation of cluster-tilting objects of $\mc_\X$ is compatible with quiver mutations and the cluster-tilting graph of $\mc_\X$ is connected. Since $\End_{\mc_{(Q,W)}}(T)$ is a $2$-CY tilted algebra for $\mc_\X$, it follows that $Q_T$ is mutation equivalent to $Q^{(2,2,2,2)}$. In particular, there is a sequence of vertices $j_1,j_2,\cdots,j_s$ of $Q_T$ such that  $Q^{(2,2,2,2)}=\mu_{j_s}\cdots\mu_{j_2}\mu_{j_1}(Q_T)$. 

Let $M$ be a basic cluster-tilting object of $\mc_\X$ such that $\End_{\mc_{(Q,W)}}(T)\cong \End_{\mc_\X}(M)$. Clearly, we may identify quiver $Q_M$ with $Q_T$. 
By Lemma~\ref{l:non-deg}, there is a  non-degenerate potential $W_M$ on $Q_M$ such that  $\End_{\mc_{(Q,W)}}(T)\cong\End_{\mc_\X}(M)\cong J(Q_M,W_{M})$ and  $\mc_{(Q_M,W_{M})}=\mc_\X$.    Denote by $M':=\mu_{j_s}\cdots\mu_{j_2}\mu_{j_1}( M)\in \mc_{\X}$ and $(Q', W')=\mu_{j_s}\cdots\mu_{j_2}\mu_{j_1}(Q_{M},W_{M})$. It is clear that $(Q', W')$ is a non-degenerate quiver with potential. According to Lemma~\ref{correspondence with mutation}, we have $Q'=Q^{(2,2,2,2)}$. By Lemma~\ref{potential on 2222}, any non-degenerate potential on $Q^{(2,2,2,2)}$ has the form  $W_{t}$ for some $t\notin\{0,1\}$ up to right equivalence. Without loss of generality, we may assume $W'=W_{t_0}$ for some $t_0\in K\backslash\{0,1\}$. By Lemma~\ref{l:gg}, we have $\End_{\mc_{\X}}(M')\cong J(Q^{(2,2,2,2)}, W_{t_0})$.

Since $(Q,W)$ is Jacobi-finite, $\Gamma(Q,W)$ is a cluster-tilting object of $\mc_{(Q,W)}$. Note that  $\End_{\mc_{(Q,W)}}(T)$ is a $2$-CY tilted algebra for $\mc_\X$, it follows that $T$ is reachable from $\Gamma(Q,W)$ by Lemma~\ref{connectedness of mc}. Consequently, there is a sequence of vertices $i_1,\cdots, i_l$ of $Q$ such that \[T=\mu_{i_l}\cdots \mu_{i_2}\mu_{i_1}(\Gamma(Q,W)) ~\text{ and}~ \End_{\mc_{(Q,W)}}(T)\cong J(\mu_{i_l}\cdots \mu_{i_2}\mu_{i_1}(Q,W))\] by Lemma~\ref{l:KY2}. 
Clearly, the quiver of $\mu_{i_l}\cdots \mu_{i_2}\mu_{i_1}(Q,W)$ is $Q_T$ and we may write $(Q_T,W'')=\mu_{i_l}\cdots \mu_{i_2}\mu_{i_1}(Q,W)$. Consequently, $\End_{\mc_\X}(M)\cong J((Q_T, W'')$. By Lemma~\ref{l:gg}, we conclude that \[\End_{\mc_\X}(M')\cong J(\mu_{j_s}\cdots\mu_{j_2}\mu_{j_1}(Q_T, W'')).\] Let us write $\mu_{j_s}\cdots\mu_{j_2}\mu_{j_1}(Q_T, W'')=(Q^{(2,2,2,2)}, W_{t_1})$ for some $t_1\in K\backslash\{0,1\}$, which is a non-degenerate quiver with potential. Hence $J(Q^{(2,2,2,2)}, W_{t_0})\cong J(Q^{(2,2,2,2)}, W_{t_1})$. By Lemma~\ref{potential on 2222} (2),  we have either $t_0=t_1$ or $t_0=t_1^{-1}$. In any case, $\mc_{(Q^{(2,2,2,2)}, W_{t_0})}\cong \mc_{(Q^{(2,2,2,2)}, W_{t_1})}$. By  Lemma~\ref{l:KY1} and Lemma~\ref{l:KY2}, we conclude that 
\[\mc_{(Q,W)}\cong \mc_{(Q_{T},W'')}\cong\mc_{(Q^{(2,2,2,2)},W_{t_1})} \cong\mc_{(Q^{(2,2,2,2)},W_{t_0})}\cong \mc_{(Q_M,W_M)}\cong \mc_{\X}.\]

 \end{proof}

\begin{remark}
If $\X$ is of type $(2,2,2,2;\lambda)$, then the condition `non-degenerate' in Theorem \ref{t:main results 3} is superfluous. 
One can prove that the quiver with potential $(Q,W)$ is non-degenerate.
Indeed, by Lemma~\ref{l:non-deg}, there is a non-degenerate potential $W_T$ on $Q_T$ such that $\mc_{(Q_T,W_T)}\cong \mc_\X$ and $J(Q_T,W_T)\cong \End_{\mc_{(Q,W)}}(T)$. Applying Lemma~\ref{l:gg}, we know that $\End_{\mc_{(Q,W)}}(\mu_k(T))\cong J(\mu_k(Q_T,W_T))$ for each vertex $k$ of $Q_T$. Consequently, $\End_{\mc_{(Q,W)}}(\mu_k(T))$ is a $2$-CY tilted algebra of $\mc_\X$. Continuing this process, we conclude that the endomorphism algebra $\End_{\mc_{(Q,W)}}(M)$ of a basic cluter-tilting $M\in \mc_{(Q,W)}$ reachable from $T$ is a $2$-CY tilted algebra of $\mc_\X$. It follows from Lemma~\ref{connectedness of mc} that $\Gamma(Q,W)$ is reachable from $T$. Hence $J(Q,W)\cong \End_{\mc_{(Q,W)}}(\Gamma(Q,W))$ is a $2$-CY tilted algebra of $\mc_\X$. Let $M\in \mc_\X$ be the basic cluster-tilting object such that $J(Q,W)\cong \End_{\mc_\X}(M)$.  We may identify $Q$ with $Q_M$. By Lemma~\ref{l:gg}, for each vertex $k$ of $Q$, we have $J(\mu_k(Q,W))\cong \End_{\mc_\X}(\mu_k(M))$. In particular, $\mu_k(Q,W)$ is $2$-acyclic. Continuing this process, we conclude that $(Q,W)$ is non-degenerate.
\end{remark}

\section{Cluster category $\mc_\X$ as stable category arising from preprojective algebra}\label{s:main-result-2}

\subsection{ 2-CY categories associated with elements in Coxeter groups}\label{s:quiver-preprojective-alg}
We follow \cite{BIRSc}.
Let $Q$ be an acyclic  quiver with  vertices $1,2,\cdots,n$. Denote by $\Lambda$  the preprojective algebra associated $Q$. For each vertex $i$, denote by $e_i$ the associated primitive idempotent
and $I_i=\Lambda(1-e_i)\Lambda$ the ideal of $\Lambda$.
Let  $W_Q$ be  the Coxeter group associated  to $Q$.
For a reduced expression $w=s_{i_1}\cdots s_{i_l}$ of $w\in W_Q$, denote by $I_w=I_{i_1}\cdots I_{i_l}$ and $\Lambda_w=\Lambda/I_w$. It is known that $I_w$ and $\Lambda_w$ are independent of the choice of reduced expressions of $w$, and $\Lambda_w$  is a finite dimensional $K$-algebra. Denote by $\Sub\Lambda_w$ the full subcategory of $\mod \Lambda_w$ whose objects are the submodules of finitely generated projective $\Lambda_w$-modules. By \cite{BIRSc}, $\Sub\Lambda_w$  is a Frobenius category and the stable category $\underline{\Sub}\Lambda_w$ is a $\Hom$-finite $2$-CY triangulated category. Moreover, $T(i_1,\ldots, i_l):=\Lambda/I_{i_1}\oplus \Lambda/I_{i_1}I_{i_2}\oplus\cdots\oplus \Lambda/I_{i_1}\cdots I_{i_l}$  is a basic cluster-tilting object in $\Sub\Lambda_w$, which  induces a basic cluster-tilting object $\overline{T}(i_1,\ldots,i_l)$ in $\underline{\Sub}\Lambda_w$.  By abuse of notations, for a fixed reduced expression $w=s_{i_1}\cdots s_{i_l}$, we also denote by $T_w:=T(i_1,\ldots, i_l)$ and $\overline{T}_w:=\overline{T}(i_1,\ldots, i_l)$.

The Gabriel quiver of $\End_{\underline{\Sub}\Lambda_w}(\overline{T}_w)$ can be constructed from the reduced expression $w=s_{i_1}\cdots s_{i_l}$ directly. Namely, we  associate with the reduced expression $w=s_{i_1}\cdots s_{i_l}$ a quiver $\widetilde{Q}(w)$ as follows, where the vertices correspond to the $s_{i_k}$.
\begin{itemize}
\item For two consecutive $i (i\in \{1, \cdots, n\})$, draw an arrow from the second one to the first one.
\item  For each edge \xymatrix{i\ar@{-}[r]^{d_{ij}}&j} of the underlying diagram of $Q$, pick out the expression consisting of the $i_k$ which are $i$ or $j$, so that we have $\cdots ii\cdots ijj \cdots jii\cdots i\cdots $. Draw $d_{ij}$ arrows from the last $i$ in a connected set of $i $ to the last $j$ in the next set of $j$, and do the same from $j$ to $i.$
\end{itemize}
Denote by $Q(w)$ the quiver obtained from $\widetilde{Q}(w)$ by removing the last $i$ for each vertex $i$ of $Q$.  
It was shown by \cite[Thm III.4.1]{BIRSc} that the Gabriel quiver  of  $\End_{\underline{\Sub}\Lambda_w}(\overline{T}(w))$  is $Q(w)$.
Moreover,  there is a rigid potential $W_{w}$ on $Q(w)$ such that $\End_{\underline{\Sub}\Lambda_w}(\overline{T}(w))\cong J(Q(w), W_{w})$ by \cite[Thm 6.4]{BIRSm}.

\begin{lem}\label{judge from quiver}
If the quiver $Q(w)$ of  $\End_{\underline{\Sub}\Lambda_w}(\overline{T}(w))$ is isomorphic to the quiver of a $2$-CY tilted algebra arising  from a weighted projective line $\X$ with at most three weights, then $\underline{\Sub}\Lambda_w$ is triangle equivalent to $\mc_\X$.
\end{lem}
\begin{proof}
By \cite[Theorem 4.4]{ART} and Lemma~\ref{l:K11}, there is a potential $W_{w}'$ on $Q(w)$ such that  $\underline{\Sub}\Lambda_w\cong \mc_{(Q(w), W_{w}')}$.  Moreover, $W_{w}'$
 is cyclically equivalent to the potential $W_{w}$ according to the proof of \cite[Proposition 3.12]{ART}. Consequently, $W_w'$ is non-degenerate by  \cite[Corollary 8.2]{DWZ}.

 Let $T$ be a basic cluster-tilting object in $\mc_\X$ such that $Q_T=Q(w)$. It follows from Lemma~\ref{l:non-deg}  that  there is a non-degenerate potential $W_T$ on $Q_T$ such that $\End_{\mc_\X}(T)\cong J(Q_T,W_T)$ and $\mc_{(Q_T,W_T)}\cong \mc_\X$.
Since   $\X$ has at most three weights,  there is a unique non-degenerate potential on $Q_{T}$ up to right equivalence by Lemma~\ref{unique potential}.  Hence  $W_T$ and $W_{w}'$ are right equivalent and \[\underline{\Sub}\Lambda_w\cong\mc_{(Q(w), W_{w}')}\cong\mc_{(Q_T,W_T)}\cong\mc_\X\]
by Lemma~\ref{l:KY1}.
\end{proof}

\subsection{Cluster category $\mc_\X$ via Buan-Iyama-Reiten-Scott's construction}

Let $\X$ be a weighted projective line with weight sequence $(p_1,p_2,p_3)$ with $p_i\ge 2$ for each $i$.  Recall that  $S_i$ is the unique simple object in the $i$-th exceptional tube of rank $p_i$ satisfying $\Hom_{\coh\X}(\co,S_i)\neq 0$. For a positive integer $j$,  $S_i^{[j]}$ is  the unique indecomposable  object in the same tube with $S_i$ such that  $S_i^{[j]}$ has length $j$ and top $S_i$.

Denote by
\[T^{(p_1,p_2,p_3)}=\co\oplus \co(\vec{c})\oplus \co(\vec{x}_1)\oplus \co(\vec{x}_2)\oplus \co(\vec{x}_3)\oplus (\bigoplus_{i=1}^3(\bigoplus_{j=1}^{p_i-2} S_{i}^{[j]} )).\]
 It is straightforward to check that $T^{(p_1,p_2,p_3)}$ is a tilting object in $\coh\X$, hence a cluster-tilting object in $\mc_\X$. 
The endomorphism algebra $B=\End_{\coh\X}(T^{(p_1,p_2,p_3)})$ is given  by  the following quiver  with relations:
\[
\xymatrix@R=0.3cm@C=0.4cm{
&&\co(\vec{x}_1)\ar@{.}[rr]\ar[ddr]&&S_{1}^{[p_1-2]}\ar[r]&S_{1}^{[p_1-3]}\ar[r]&\cdots\ar[r]&S_{1}^{[1]}&&\\
&&\co(\vec{x}_2)\ar@{.}[rr]\ar[dr]&&S_{2}^{[p_2-2]}\ar[r]&S_{2}^{[p_2-3]}\ar[r]&\cdots\ar[r]&S_{2}^{[1]}&&\\
&\co\ar[uur]\ar[ur]\ar[dr]\ar@{.}[rr]&&\co(\vec{c})\ar[uur]\ar[ur]\ar[dr]&&&&&&\\
&&\co(\vec{x}_3)\ar@{.}[rr]\ar[ur]&&S_{3}^{[p_3-2]}\ar[r]&S_{3}^{[p_3-3]}\ar[r]&\cdots\ar[r]&S_{3}^{[1]}.&&
}
\]
According to Section~\ref{s:keller} and \ref{s:gen-cluster-cat-tilting-quiver},  the Gabriel quiver ${Q}_{T^{(p_1,p_2,p_3)}}$ of $\End_{\mc_\X}(T^{(p_1,p_2,p_3)})$ is as follows:
\[
\xymatrix@R=0.3cm@C=0.4cm{
&&\co(\vec{x}_1)\ar[ddr]&&S_{1}^{[p_1-2]}\ar[ll]\ar[r]&S_{1}^{[p_1-3]}\ar[r]&\cdots\ar[r]&S_{1}^{[1]}&&\\
&&\co(\vec{x}_2)\ar[dr]&&S_{2}^{[p_2-2]}\ar[ll]\ar[r]&S_{2}^{[p_2-3]}\ar[r]&\cdots\ar[r]&S_{1}^{[1]}&&\\
{Q}_{T^{(p_1,p_2,p_3)}}:&\co\ar[uur]\ar[ur]\ar[dr]\ar@{.}[rr]&&\co(\vec{c})\ar[ll]\ar[uur]\ar[ur]\ar[dr]&&&&&&\\
&&\co(\vec{x}_3)\ar[ur]&&S_{3}^{[p_3-2]}\ar[ll]\ar[r]&S_{3}^{[p_3-3]}\ar[r]&\cdots\ar[r]&S_{3}^{[1]}.&&
}
\]

\begin{thm}\label{t:weighted via preprojective}
Let $\X$ be a weighted projective line with at most three weights and $\mc_\X$ the associated cluster category. There is an acyclic quiver $Q$ and a reduced word $w\in W_Q$ such that $\mc_\X$ is triangle equivalent to  $\underline{\Sub}\Lambda_w$.
\end{thm}
\begin{proof}
Let us assume that $\X$ has at most two weights or $\X$  has three weights but at least two of them are $2$.  In particular, $\X$ is of domestic type.  In this case,  $\coh\X$ is derived equivalent to the module category of  $KQ$ for some acyclic quiver $Q$ of affine type. By \cite[Thm III.3.4]{BIRSc},  each cluster category $\mc_{KQ}$ can be realized as the stable category $\underline{\Sub}\Lambda_w$ for some reduced word $w\in W_Q$.
Therefore $\mc_\X\cong\mc_{KQ}\cong \underline{\Sub}\Lambda_w$.

Now assume that $\X$ has three weights $(p_1,p_2,p_3)$ and at most one of them is $2$.  Up to permutation, we may assume that $p_1\leq p_2\leq p_3$. 

Suppose that $p_1=2$.  By assumption,  $p_2>2$ and $p_3>2$. Let $Q$ be the following quiver
\[
\xymatrix@R=0.3cm@C=0.4cm{
&&b_1\ar[r]&b_2\ar[r]&\cdots\ar[r]&b_{p_2-2}&&\\
a_1\ar[r]&o\ar[ur]\ar[dr]&&&&&&\\
&&c_1\ar[r]&c_2\ar[r]&\cdots\ar[r]&c_{p_3-2}.&&
}
\]
Let $w=s_os_{b_1}s_{c_1}s_{a_1}s_os_{b_1}\cdots s_{b_{p_2-2}}s_{c_1}\cdots s_{c_{p_3-2}}s_{a_1}s_os_{b_1}\cdots s_{b_{p_2-2}}s_{c_1}\cdots s_{c_{p_3-2}}$ be a word in $W_{Q}$. 
It follows from \cite[Thm 2.3]{ORT} that $w$ is a reduced expression. Applying the construction in Section~\ref{s:quiver-preprojective-alg}, we obtain the quiver $Q(w)$ of $\End_{\underline{\Sub}\Lambda_{w}}(\overline{T}_{w})$ as follows
\[
\xymatrix@R=0.3cm@C=0.4cm{
&& {a_1}\ar[ddr]&&&&&&&\\
&& {b_1}\ar[dr]&& {b_1}\ar[r]\ar[ll]& {b_2}\ar[r]&\cdots\ar[r]& {b_{p_2-2}}&&\\
{Q(w)}:& o\ar[uur]\ar[ur]\ar[dr]&& o\ar[ur]\ar[ll]\ar[dr]&&&&&&\\
&& {c_1}\ar[ur]&& {c_1}\ar[ll]\ar[r]& {c_2}\ar[r]&\cdots\ar[r]& {c_{p_3-2}.}&&
}
\]
In particular, $Q(w)=Q_{T^{(2, p_2,p_3)}}$.  Consequently,   $\mc_\X\cong \underline{\Sub}\Lambda_{w}$ by Lemma~\ref{judge from quiver}.

Now suppose that  $p_1\ge 3$. Let $Q$ be the following quiver 
\[
\xymatrix@R=0.3cm@C=0.4cm{
&a_1\ar[r]&a_2\ar[r]&\cdots\ar[r]&a_{p_1-2}&&\\
&b_1\ar[r]&b_2\ar[r]&\cdots\ar[r]&b_{p_2-2}&&\\
o\ar[uur]\ar[ur]\ar[dr]&&&&&&\\
&c_1\ar[r]&c_2\ar[r]&\cdots\ar[r]&c_{p_3-2}.&&
}
\]
Let $W_Q$ be the Coxeter group of $Q$, define \[w=s_os_{a_1}s_{b_1}s_{c_1}s_os_{a_1}\cdots s_{p_1-2}s_{b_1}\cdots s_{b_{p_2-2}}s_{c_1}\cdots s_{c_{p_3-2}}s_os_{a_1}\cdots s_{p_1-2}s_{b_1}\cdots s_{b_{p_2-2}}s_{c_1}\cdots s_{c_{p_3-2}}\in W_Q.\] 
Again by \cite[Thm 2.3]{ORT}, $w$ is a reduced word. Similarly, we obtain the quiver  $Q(w)$ of $\End_{\underline{\Sub}\Lambda_{w}}(\overline{T}_{w})$ as follows
 
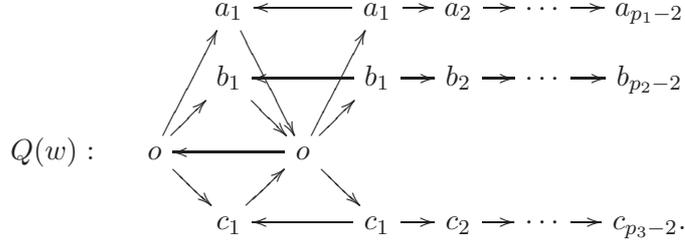
\[
\xymatrix@R=0.3cm@C=0.4cm{
&&{a_1}\ar[ddr]&& {a_1}\ar[ll]\ar[r]& {a_2}\ar[r]&\cdots\ar[r]& {a_{p_1-2}}&&\\
&& {b_1}\ar[dr]&& {b_1}\ar[ll]\ar[r]& {b_2}\ar[r]&\cdots\ar[r]& {b_{p_2-2}}&&\\
{Q(w)}:& o\ar[uur]\ar[ur]\ar[dr]&& o\ar[ll]\ar[uur]\ar[ur]\ar[dr]&&&&&&\\
&& {c_1}\ar[ur]&& {c_1}\ar[ll]\ar[r]& {c_2}\ar[r]&\cdots\ar[r]& {c_{p_3-2}.}&&
}
\]
In particular,  $Q(w)={Q}_{T^{(p_1,p_2,p_3)}}$. By  Lemma~\ref{judge from quiver}, $\mc_\X\cong \underline{\Sub}\Lambda_{w}$.
\end{proof}

\begin{cor}
Let $\X$ be a weighted projective line with  at most three weights. Then each $2$-CY tilted algebra arising from $\mc_\X$ is a Jacobian  algebra of a quiver with rigid potential.
\end{cor}
\begin{proof}
By Theorem~\ref{t:weighted via preprojective}, there is an acyclic quiver $Q$ and a reduced word $w\in W_Q$ such that $\mc_\X$ is triangle equivalent to $\underline{\Sub}\Lambda_w$. Moreover, there is a basic cluster-tilting object $T$ in $\mc_\X$ such that $Q_T=Q(w)$ and $\End_{\mc_\X}(T)\cong \End_{\underline{\Sub}\Lambda_w}(\overline{T}_w)$. 
By Lemma~\ref{l:non-deg},  $\End_{\mc_\X}(T)\cong J(Q_T,W_T)$ for  a non-degenerate  potential $W_T$ on $Q_T$. On the other hand, there is rigid potential $W_w$ on $Q(w)$ such that $\End_{\underline{\Sub}\Lambda_w}(\overline{T}_w)\cong J(Q(w),W_w)$ by \cite[Thm 6.4]{BIRSm}.

Since  $\X$ has at most three weights, by Lemma~\ref{unique potential}, there is a unique non-degenerate potential on $Q_{T}$ up to right equivalence.  Consequently, $W_T$ is right equivalent to $W_w$. It is known that the mutation of a rigid potential is rigid. For any vertex $i$ of $Q_T$, we obtain $\End_{\mc_\X}(\mu_i(T))\cong J(\mu_i(Q_T,W_w))$ by Lemma~\ref{l:KY2}.  
Now the result follows from \cite[Thm 1.2]{FG} that each basic cluster-tilting object is reachable from $T$.

\end{proof}

\end{document}